\documentclass[11pt,twoside]{article}

\usepackage{amsfonts}
\usepackage{amsmath}
\usepackage{amsthm}
\usepackage[mathscr]{euscript}
\usepackage{mathrsfs}
\usepackage{indentfirst}
\usepackage{booktabs,multirow,bigstrut,makecell}
\usepackage{color}\usepackage{enumitem}

\newtheorem{theorem}{Theorem}[section]
\newtheorem{lemma}[theorem]{Lemma}
\newtheorem{corollary}[theorem]{Corollary}

\newtheorem{remark}[theorem]{Remark}
\newtheorem{proposition}[theorem]{Proposition}
\numberwithin{equation}{section}

\newcommand{\lbl}[1]{\label{#1}}
\allowdisplaybreaks

\newcommand{\be}{\begin{equation}}
\newcommand{\ee}{\end{equation}}
\newcommand\bes{\begin{eqnarray}} \newcommand\ees{\end{eqnarray}}
\newcommand{\bess}{\begin{eqnarray*}}
\newcommand{\eess}{\end{eqnarray*}}
\newcommand{\bbbb}{\left\{\begin{aligned}}
\newcommand{\nnnn}{\end{aligned}\right.}
\newcommand{\bea}{\begin{align*}}
\newcommand{\eea}{\end{align*}}

\newcommand\kk{\left}
\newcommand\rr{\right}
\newcommand\dd{\displaystyle}

\newcommand\df{\dd\frac}

\newcommand\yy{\infty}
\newcommand\qq{\eqref}

\newcommand\R{\mathbb{R}}
\newcommand\ol{\overline}

\newcommand\sk{\smallskip}
\newcommand\ri{\rightarrow}
\newcommand\mk{\medskip}

\begin{document}
\setlength{\baselineskip}{16pt} \pagestyle{myheadings}

\begin{center}{\Large\bf Free boundary problems with nonlocal and local diffusions}\\[2mm]
{\Large\bf II: Spreading-vanishing and long-time behavior\footnote{This work was supported by NSFC Grants 11771110, 11971128}}\\[4mm]
  {\Large  Jianping Wang, \ Mingxin Wang\footnote{Corresponding author. {\sl E-mail}: mxwang@hit.edu.cn}}\\[1.5mm]
School of Mathematics, Harbin Institute of Technology, Harbin 150001, PR China
\end{center}

\date{\today}

\begin{abstract}This is part II of our study on the free boundary problems with nonlocal and local diffusions. In part I, we obtained the existence, uniqueness, regularity and estimates of global solution. In part II here, we show a spreading-vanishing dichotomy, and provide the criteria of spreading and vanishing, as well as the long time behavior of solution when spreading happens.



\end{abstract}

\section{Introduction}
\setcounter{equation}{0} {\setlength\arraycolsep{2pt}
\markboth{J. P. Wang and M. X. Wang}{Free boundary problems with nonlocal and local diffusions}

We continue our investigation in \cite{WW-I} on the free boundary problems with nonlocal and local diffusions
 \bes
\left\{\begin{aligned}
&u_t=d_1\int_{g(t)}^{h(t)}\!J(x-y)u(t,y){\rm d}y-d_1u+f_1(u,v), &&t>0,~g(t)<x<h(t),\\
&v_t=d_2 v_{xx}+f_2(u,v), &&t>0,~g(t)<x<h(t),\\
&u(t,g(t))=u(t,h(t))=v(t,g(t))=v(t,h(t))=0, &&t\ge 0,\\
&h'(t)=-\mu v_x(t,h(t))+\rho\int_{g(t)}^{h(t)}\!\!\int_{h(t)}^\infty\! J(x-y)u(t,x){\rm d}y{\rm d}x, &&t\ge 0,\\[1mm]
&g'(t)=-\mu v_x(t,g(t))-\rho\int_{g(t)}^{h(t)}\!\!\int_{-\infty}^{g(t)}\! J(x-y)u(t,x){\rm d}y{\rm d}x, &&t\ge 0,\\
&u(0,x)=u_0(x)>0, \ v(0,x)=v_0(x)>0,\  &&-h_0<h<h_0,\\
&h(0)=-g(0)=h_0>0,
\end{aligned}\right.
 \label{1.1}
 \ees
where $[-h_0,h_0]$ represents the initial population range of the species $u$ and $v$; $x=g(t)$ and $x=h(t)$ are the free boundaries to be determined together with $u(t,x)$ and $v(t,x)$, which are always assumed to be identically $0$ for $x\in \mathbb{R}\setminus [g(t), h(t)]$; $d_i$ and $\mu,\rho$ are positive constants. The kernel function $J: \mathbb{R}\rightarrow\mathbb{R}$ is continuous and satisfies
 \begin{enumerate}[leftmargin=3em]
\item[\bf(J1)]  $J\in C^{1-}(\R)$, $J(0)>0,~J(x)\ge 0, \ \dd\int_{\mathbb{R}}J(x){\rm d}x=1$, \ $J$\, is\, symmetric, \ and\ $\dd\sup_{\mathbb{R}}J<\infty$,
 \end{enumerate}
where $J\in C^{1-}(\R)$ means that $J$ is Lipschitz continuous in $\mathbb{R}$. Reaction terms $f_1,f_2$ are either classical competition model:
\bes
 f_1(u,v)=u(a-u-bv), \ \ f_2(u,v)=v(1-v-cu),
  \lbl{1a.2}\ees
or classical prey-predator model:
 \bes
 f_1(u,v)=u(a-u-bv), \ \ f_2(u,v)=v(1-v+cu),
  \lbl{1a.3}\ees

It follows from {\bf(J)} that there exist constants $\bar\varepsilon\in(0,h_0/4)$ and $\delta_0>0$ such that
 \bes
J(x,y)>\delta_0\ \ \ {\rm if}\ \ |x-y|<\bar\varepsilon.
 \label{1.5}\ees

Denote by $C^{1-}(\Omega)$ the Lipschitz continuous function space in $\Omega$. Under the conditions:
 \bess
(u_0,v_0)\in C^{1-}([-h_0,h_0])\times W^2_p(-h_0,h_0) \ \mbox{with} \ p>3, \ u_0(\pm h_0)=v_0(\pm h_0)=0,
\label{1.4}
 \eess
it has been proved in the first part (\cite{WW-I}) that \eqref{1.1} has a unique global solution $(u,v,g,h)$:
 \bes
 0< u\le k_1,\ \ 0<v\le k_2,\ \ g'(t)<0, \ \ h'(t)>0, \ \ 0<-v_x(t,h(t)),\ v_x(t,g(t))\le k_3,
 \lbl{2.1}\ees
where
 \[k_1=\max\left\{\|u_0\|_\infty,\,a\right\}, \ \ k_3=\max\left\{\frac 1{h_0}, \ \sqrt{\frac{L}{2d_2}}, \
 \frac {\|v_0'\|_{C([-h_0,h_0])}}{k_2}\right\}, \ \ L=\sup_{(0,k_1)\times(0,k_2)}f_2(u,v),\]
and
 \bess
 k_2=\max\left\{\|v_0\|_\infty,\,1\right\} \ \ \mbox{when\ \eqref{1a.2}\ hold}, \ \ \
 k_2=\max\left\{\|v_0\|_\infty,\,1+ck_1\right\} \ \ \mbox{when\ \eqref{1a.3}\ hold}.
 \eess
Moreover, for any given $0<\tau<T<\yy$, we have
 \[ g,\,h\in C^{1+\alpha/2}([0,T]), \ \  u\in C^{1,1-}(\ol D^T_{g,h}), \ \ v\in C^{1+\alpha/2,\,2+\alpha}([\tau,T]\times[g(t),h(t)]),\]
where
 \[D^T_{g,h}=\left\{(t,x)\in\mathbb{R}^2:\, 0<t\leq T,~g(t)<x<h(t)\right\},\]
$u\in C^{1,1-}(\ol D^T_{g,h})$ means that $u$ is differentiable continuously in $t\in[0,T]$ and is Lipschitz continuous in $x\in[g(t),h(t)]$.

\vskip 4pt
In view of \eqref{2.1} we can define
 \[ \lim_{t\to\infty}g(t)=g_\infty\in[-\infty,-h_0), \ \ \ \lim_{t\to\infty}h(t)=h_\infty\in(h_0,\infty].\]
 Clearly we have either
 \[\mbox{(i) $h_\infty-g_\infty<\infty$, \ \ \ or \ (ii) $h_\infty-g_\infty=\infty$. }\]
 We will call (i) the vanishing case, and call (ii) the spreading case.

The main aims of this part are concerned with the spreading-vanishing dichotomy, the criteria of spreading and vanishing, as well as the long-time behavior of solution when spreading happens. The main results of this part are the following theorems.

\begin{theorem}[Spreading-vanishing dichotomy]\label{th1.1}\, Let $(u,v,g,h)$ be the unique solution of \eqref{1.1}. Then either
\begin{enumerate}[leftmargin=3em]
\item[{\rm(i)}] {\rm Spreading:} $h_\infty-g_\infty=\infty$,
  \end{enumerate}
or
 \begin{enumerate}[leftmargin=3em]
\item[{\rm(ii)}] {\rm Vanishing:} $h_\infty-g_\infty<\infty$ and
    \[\lim_{t\to\infty}\max_{g(t)\le x\le h(t)}u(t,x)=\lim_{t\to\infty}\max_{g(t)\le x\le h(t)}v(t,x)=0.\]
\end{enumerate}
\end{theorem}

To determine the long-time behavior of the solution when spreading happens, we restrict to two special cases:
\begin{itemize}
\item[(a)] \underline{The weak competition case}: $(f_1, f_2)$ satisfies \eqref{1a.2} with $1/{c}>a>b$.
\item[(b)] \underline{The weak predation case}: $(f_1, f_2)$ satisfies \eqref{1a.3} with $a>b+abc$.
\end{itemize}

\begin{theorem}[Long-time behavior]\label{{th1.2}}\, Let $(u,v,g,h)$ be the unique solution of \eqref{1.1}   and $\dd\lim_{t\to\infty} (g(t), h(t))=\mathbb R$, i.e., spreading happens.

{\rm(i)} in the weak competition case we have
\bess
\lim_{t\rightarrow\infty}(u(t,x),v(t,x))=\kk(\df{a-b}{1-bc},\df{1-ac}{1-bc}\rr)\ \ {\rm locally\ uniformly\ for}\ x\in\R,
\eess

{\rm(ii)} in the weak predation case we have
\bess
\lim_{t\rightarrow\infty}(u(t,x),v(t,x))=\kk(\df{a-b}{1+bc},\df{1+ac}{1+bc}\rr)\ \ {\rm locally\ uniformly\ for}\ x\in\R.
\eess
\end{theorem}

We remark that, the spreading-vanishing dichotomy and long-time behavior when spreading happens are  parallel to those of the local system (\cite{GW12, WZjdde14, Wcnsns15}) and nonlocal system (\cite{DWZ19}). Unfortunately, we have to leave the spreading speeds of the moving boundaries $g,\,h$ when spreading happens as open issue.

\begin{theorem}[Spreading-vanishing criteria]\label{th1.3}
Under the conditions of Theorem \ref{th1.1}. If one of the following holds:
 \vspace{-2mm}
\begin{enumerate}[leftmargin=3em]
\item[{\rm(i)}]\, $a\ge d_1$,\vspace{-2mm}

\item[{\rm(ii)}]\, $h_0\ge \frac 12\pi\sqrt{d_2}$,\vspace{-2mm}

\item[{\rm(iii)}]\, $a<d_1$ and $h_0\ge \ell^*/2$, where $\ell^*$ satisfies $\lambda_p(\mathcal{L}_I+a)=0$ when  $|I|=\ell^*$,
 \vspace{-2mm}\end{enumerate}
then spreading always happens.

If $a<d_1$ and $h_0<\frac 12\min\{\pi\sqrt{d_2},\ell^*\}$, then there is $\Lambda^*\ge\Lambda_*>0$ such that $h_\infty-g_\infty<\infty$ when $\mu+\rho\le\Lambda_*$, and $h_\infty-g_\infty=\infty$ when $\mu+\rho\ge\Lambda^*$.
\end{theorem}

Here we mention that, same as the single equation in \cite{CDLL}, nonlocal diffusion will change the spreading-vanishing criteria. For the corresponding local diffusive competition and prey-predator models, from the results of \cite{GW12, WZjdde14, Wcnsns15, WZhang16} we see that no matter how small is the diffusion coefficient $d_1$ in $d_1u_{xx}$ relative to $a$, vanishing can always happen if $h_0$ and $\mu,\,\rho$ are sufficiently small. However, for the nonlocal and local diffusions problem \eqref{1.1}, Theorem \ref{th1.3} shows that when $a\ge d_1$, spreading always happens no mater how small $h_0$, $\mu,\,\rho$, $u_0$ and $v_0$ are. Moreover, we find a new critical value $\frac12\min\{\pi\sqrt{d_2},\ell^*\}$ which plays an important role in governing the spreading and vanishing phenomenon.

This paper is arranged as follows. Section 2 gives some preliminary results. Section 3 is devoted to the spreading-vanishing dichotomy and long-time behavior when spreading happens. The criteria governing spreading and vanishing will be given in Section 4.

\section{Preliminaries}
\setcounter{equation}{0} {\setlength\arraycolsep{2pt}

For the given $T>0$, we define
 \bess
\mathbb H^T&=&\left\{h\in C^1([0,T])~:~h(0)=h_0,
\; 0<h'(t)\le R(t)\right\},\\[.1mm]
\mathbb G^T&=&\left\{g\in C^1([0,T])~:-g\in\mathbb{H}^T\right\}.
 \eess

\subsection{ Maximum principle and comparison principle}

\begin{lemma}[Maximum Principle {\cite[Lemma 2.2]{CDLL}}]\label{l2.2}
Assume that $J$ satisfies {\bf(J)} and $d$ is a positive constant, and $(r, \eta)\in\mathbb G^T\times\mathbb H^T$. Suppose that $\psi, \psi_t\in C(\overline D^T_{\eta,r})$ and fulfill, for some $\varrho\in L^\infty (D^T_{\eta,r})$,
 \bess\left\{\begin{aligned}
&\psi_t\ge d\int_{\eta(t)}^{r(t)}\!J(x,y)\psi(t,y){\rm d}y-d\psi+\varrho\psi, && (t,x)\in D^T_{\eta,r},\\
& \psi(t, \eta(t)) \geq 0,\ \psi(t, r(t)) \geq 0, && 0\le t\le T,\\
&\psi(0,x)\ge0,  && |x|\le h_0.
 \end{aligned}\right.\eess
Then $\psi\ge0$ on $\overline D^T_{\eta, r}$. Moreover, if $\psi(0,x)\not\equiv0$ in $[-h_0, h_0]$, then $\psi>0$ in $D^T_{\eta,r}$.
\end{lemma}

\begin{lemma}[Maximum principle {\cite[Lemma 3.3]{CDLL}}]\label{l3.3} For $T>0$, set $\Omega_0:=(0, T]\times [-h_0, h_0]$. Suppose $u,\,u_t\in C(\ol\Omega_0)$ and for some $\varrho\in L^\infty(\Omega_0)$,
 \bess \left\{\begin{aligned}
&u_t\ge d\int_{-h_0}^{h_0}\!J(x-y)u(t,y){\rm d}y-d u+\varrho u, && (t,x)\in\Omega_0,\\
&u(0,x)\ge0,  &&|x|\le h_0.
  \end{aligned}\right.\eess
Then $u\ge0$ on $\ol\Omega_0$. Moreover, if $u(0,x)\not\equiv0$ in $[-h_0, h_0]$, then $u>0$ in $\Omega_0$.
\end{lemma}

\begin{lemma}[Comparison principle]\label{l3.4}\, Let $T>0$, $\bar h,\,\bar g\in C([0,T])$, $\bar u\in C(\overline D^T_{\bar g,\bar h})$, $\bar v\in W^{1,2}_p(D^T_{\bar g,\bar h})$ with $p>3$ and satisfy
 \bes\label{3.2}
 \left\{\begin{aligned}
&\bar u_t\dd\ge d_1\int_{\bar g(t)}^{\bar h(t)}\!\!J(x-y)\bar u(t,y){\rm d}y
-d_1\bar u+\bar u(a-\bar u), &&(t,x)\in D^T_{\bar g,\bar h},\\
&\bar v_t\ge d_2\bar v_{xx}+\bar v(1-\bar v), &&(t,x)\in D^T_{\bar g,\bar h},\\
 &\bar u(t,\bar g(t))\ge 0,\ \
 \bar v(t,\bar h(t))\geq 0,\ \ && 0<t\le T,\\
&\bar h'(t)\ge-\mu\bar v_x(t,h(t))+\rho\int_{\bar g(t)}^{\bar h(t)}\!\!\int_{\bar h(t)}^{\infty}\!
 J(x-y)\bar u(t,x){\rm d}y{\rm d}x, &&0<t\le T,\\[1mm]
&\bar g'(t)\le-\mu\bar v_x(t,g(t))-\rho\int_{\bar g(t)}^{\bar h(t)}\!\!\int_{-\infty}^{\bar g(t)}\!J(x-y)\bar u(t,x){\rm d}y{\rm d}x, &&0<t\le T,\\
  &\bar u(0,x)\ge u_0(x),\ \ \bar v(0,x)\ge v_0(x); && |x|\le h_0,\\
  &\bar h(0)\ge h_0,~\ \bar g(0)\le-h_0.
 \end{aligned}\right.\ees
Let $(u, g, h)$ be the unique solution of \eqref{1.1} in there $(f_1, f_2)$ satisfies
\eqref{1a.2}. Then we have
  \[u\leq\bar u,~ \ v\leq\bar v,~ \ g\geq\bar g,~\ h\leq\bar h~\ \ \mbox{in} \ \ D^T_{g,h}.\]
\end{lemma}

 \begin{lemma}[Comparison principle]\label{l3.5}\,In Lemma $\ref{l3.4}$, if we replace the second inequality of \eqref{3.2} by
 \[\bar v_t\ge d_2\bar v_{xx}+\bar v(1-\bar v+c\bar u), \ \ (t,x)\in D^T_{\bar g,\bar h},\]
and let $(u, v, g, h)$ be the unique solution of \eqref{1.1} in there $(f_1, f_2)$ satisfies \eqref{1a.3}, then the conclusion is still true.
 \end{lemma}

Lemmas \ref{l3.4} and \ref{l3.5} can be proved by the combination of proofs of
\cite[Lemma 5.1]{GW12}, \cite[Theorem 3.1]{CDLL} and \cite[Lemma 3.1]{WZjdde17}.
We omit the details here.

\subsection{Some related eigenvalue problems}

Here we recall some results on the principal eigenvalue of linear
operator $\mathcal{L}_{\Omega}+\theta: C(\bar\Omega)\mapsto C(\bar\Omega)$ defined by
 $$\left(\mathcal{L}_{\Omega}+\theta\right)\varphi:= d\left(\int_{\Omega}J(x-y)
 \varphi(y){\rm d}y-\varphi\right)+\theta(x)\varphi,$$
where $\Omega$ is an open interval in $\mathbb{R}^n$, possibly unbounded,
$\theta\in C(\bar\Omega)$ and $J$ satisfies the condition {\bf(J)}.

Define the generalized principal eigenvalue of $\mathcal{L}_{\Omega}+\theta$:
  $$\lambda_p(\mathcal{L}_{\Omega}+\theta):=\inf\ \Big\{\lambda\in\mathbb R:  (\mathcal{L}_{\Omega}+\theta)\varphi\leq\lambda\varphi \ \mbox{ in }\Omega \ \mbox{ for some } \
  \varphi\in C(\bar\Omega),\,\varphi>0\Big\}.$$
As usual, if $\lambda_p(\mathcal{L}_{\Omega}+\theta)$ has a continuous and positive eigenfunction, i.e., there exists a continuous and positive function $\varphi_p$ such that $(\mathcal{L}_\Omega+\theta)\varphi_p=\lambda_p(\mathcal{L}_{\Omega}+\theta)\varphi_p(x)$, we call $\lambda_p(\mathcal{L}_{\Omega}+\theta)$ a {\it principal eigenvalue} of $\mathcal{L}_{\Omega}+\theta$.

Using the variational characterization of $\lambda_p(\mathcal{L}_{\Omega}+\theta)$
(see, e.g., \cite{7-BJFA16}):
 \[\lambda_p(\mathcal{L}_{\Omega}+\theta)=\sup_{0\not\equiv \varphi\in L^2(\Omega)}
 \frac{d\dd\int_\Omega\int_\Omega J(x-y)\varphi(y)\varphi(x){\rm d}y{\rm d}x+\int_\Omega(\theta(x)-d)\varphi^2(x){\rm d}x}
 {\dd\int_\Omega\varphi^2(x){\rm d}x},\]
we can show that $\lambda_p(\mathcal{L}_{\Omega}+\theta)$ is strictly increasing in $\theta(x)$,
i.e., $\theta_1(x)\le \theta_2(x)$ and $\theta_1(x)\not\equiv\theta_2(x)$ implies
 \[\lambda_p(\mathcal{L}_{\Omega}+\theta_1)<\lambda_p(\mathcal{L}_{\Omega}+\theta_2).\]
  \begin{enumerate}[leftmargin=3em]
\item[{\bf(F)}] $f\in C(\mathbb R\times[0,\infty))$ and is differentiable with respect to $u$, $f_u(\cdot,0)$ is locally Lipschitz continuous in $\mathbb{R}$, $f(\cdot,0)\equiv 0$ and $f(x, u)/u$ is strictly decreasing with respect to $u\in\mathbb{R}^+$. Moreover, there exists $M>0$ such that $f(x,u)<0$ for all $u\ge M$ and all $x\in \mathbb{R}$.
\end{enumerate}

We consider the problem
 \bes\label{3.3}
 \left\{\begin{aligned}
 &u_t=d\dd\left(\int_\Omega J(x-y)u(t,y){\rm d}y-u\right)+f(x, u),&&  t>0,~x\in\Omega,\\
  &u(0,x)=u_0(x),\ \  && x\in\Omega.
\end{aligned}\right.
\ees

\begin{proposition}{\rm(\cite{10-Bjmaa07, 19-Cjde10})}\label{p3.4}\,Suppose {\bf(J)}
and {\bf(F)} hold. Then \eqref{3.3} admits a unique positive steady state $u_{\Omega}\in C(\bar\Omega)$ if and only if
  $$\lambda_p(\mathcal{L}_{\Omega}+ f_u(x,0))>0.$$
Moreover, for $u_0\in C(\bar\Omega)$ and $u_0\ge,\,\not\equiv0$, the unique solution $u(t,x)$ of \eqref{3.3} satisfies $\dd\lim_{t\to\infty}\|u(t,\cdot)-u_{\Omega}(\cdot)\|_{C(\bar\Omega)}=0$ if $\lambda_p(\mathcal{L}_{\Omega}+f_u(x,0))>0$, while $\dd\lim_{t\to\infty}\|u(t,\cdot)\|_{C(\bar\Omega)}=0$ if $\lambda_p(\mathcal{L}_{\Omega}+ f_u(x,0))\leq 0$.\end{proposition}

\begin{proposition}{\rm(\cite[Proposition 3.4]{CDLL})}\label{p3.7}\, Assume that
the condition {\bf(J)} holds, $\theta_0$ is a constant and $-\infty<\ell_1<\ell_2
<\infty$. Then the following hold true:

\sk{\rm(i)}\, $\lambda_p(\mathcal{L}_{(\ell_1, \ell_2)}+\theta_0)$ is strictly increasing and
continuous in $\ell:=\ell_2-\ell_1$,

\mk{\rm(ii)}\, $\dd\lim_{\ell_2-\ell_1\to\infty}\lambda_p(\mathcal{L}_{(\ell_1, \ell_2)}+  \theta_0)=\theta_0$,

\mk{\rm(iii)}\, $\dd\lim_{\ell_2-\ell_1\to 0}\lambda_p(\mathcal{L}_{(\ell_1, \ell_2)}+  \theta_0)=\theta_0-d$.
\end{proposition}

\begin{remark}\lbl{r3.6}\,Since $\theta_0$ is a constant, it follows easily from the
definition that $\lambda_p(\mathcal{L}_{(\ell_1, \ell_2)}+\theta_0)$ depends
only on $\ell:=\ell_2-\ell_1$, i.e.,
 \[\lambda_p(\mathcal{L}_{(\ell_1, \ell_2)}+\theta_0)=\lambda_p(\mathcal{L}_{(0, \ell)}+\theta_0) \ \ \
  {\rm with } \ \ell:=\ell_2-\ell_1.\]
 \end{remark}

\section{Spreading-vanishing dichotomy and long-time behavior}
\setcounter{equation}{0} {\setlength\arraycolsep{2pt}

To establish the spreading-vanishing dichotomy we first give some abstract propositions. Let $\dd\lim_{t\to\infty}h(t)=h_\infty$ and $\dd\lim_{t\to\infty}g(t)=g_\infty$. Clearly, $h_\infty,-g_\infty\le\infty$. In what follows, we always suppose that $f_1,f_2$ satisfy \eqref{1a.2} or \eqref{1a.3}.

\begin{lemma}{\rm(\cite[Proposition 2]{WZ-dcdsa18})}\label{l4.a}\, Let $d$, $C$, $\mu$ and $\eta_0$ be positive constants, $w\in W^{1,2}_p((0,T)\times(0,\eta(t)))$ and $w_0\in W^2_p(0,\eta_0)$ for some $p>1$ and any $T>0$, and $w_x\in C([0,\infty)\times(0,\eta(t)])$, $\eta\in C^1([0,\infty))$. If $(w,\eta)$ satisfies
  \bess\left\{\begin{array}{lll}
 w_t-d w_{xx}\geq -C w, &t>0,\ \ 0<x<\eta(t),\\[.5mm]
 w\ge 0,\ \ \ &t>0, \ \ x=0,\\[.5mm]
 w=0,\ \ \eta'(t)\geq-\mu w_x, \ &t>0,\ \ x=\eta(t),\\[.5mm]
 w(0,x)=w_0(x)\ge,\,\not\equiv 0,\ \ &x\in (0,\eta_0),\\[.5mm]
 \eta(0)=\eta_0,
 \end{array}\right.\eess
and $\dd\lim_{t\to\infty} \eta(t)=\eta_\infty<\infty$, $\dd\lim_{t\to\infty} \eta'(t)=0$,
   \bess
 \|w(t,\cdot)\|_{C^1([0,\,\eta(t)])}\leq M, \ \, \forall\, t>1
 \eess
for some constant $M>0$. Then $\dd\lim_{t\to\infty}\,\max_{0\leq x\leq \eta(t)}w(t,x)=0$.
\end{lemma}

The following lemma provides an estimate for the solution component $v$.
\begin{lemma}\label{l4.1}
Let $(u,v,g,h)$ be the unique global solution of \eqref{1.1} and $h_\infty-g_\infty<\infty$ and $D_\infty=[0,\infty)\times[g(t),h(t)]$. Then there is $C>0$ such that
 \bes
\|v\|_{C^{\frac{1+\alpha}{2},1+\alpha}(D_\infty)}\le C,
  \label{4.1}\ees
and hence
\bes
\|v_x(t,h(t))\|_{C^{\frac \alpha 2}(\overline{\mathbb R}_+)}+\|v_x(t,g(t))\|_{C^{\frac \alpha 2}(\overline{\mathbb R}_+)}\le C.\label{4.2}
\ees
\end{lemma}

\begin{proof}
Since the proof is similar to that of \cite[Theorem 2.1]{Wjfa16} and \cite[Theorem 2.2]{WZjde18}, we give the sketch of the proof and omit the details. It is easy to derive from \eqref{2.1} that
\bes
0< h'(t),-g'(t)\le \mu k_3+\rho k_1(h_\infty-g_\infty)=:C_0<\infty.\label{4.3a}
\ees
We straighten the free boundary. Similar to the above, set
$w(t,y)=u(t,x(t,y))$, $z(t,y)=v(t,x(t,y))$. Then
\bess
\left\{\begin{aligned}
&z_t=d_2\xi(t)z_{yy}+\zeta(t,y)z_y+f_2(w,z), &&t>0,~|y|<1,\\
&z(t,-1)=z(t,1)=0, &&t\ge0,\\
&z(0,y)=v_0(h_0y)=:z_0(y), &&|y|\le 1.
\end{aligned}\right.
 \eess
Due to \eqref{4.3a}, it is easy to get
$$\|\xi\|_{L^\infty(\mathbb{R}^+)}\le 1/h_0^2,\ \ \ \|\zeta\|_{L^\infty(\mathbb{R}^+\times \Sigma)}\le 2C_0/h_0, \ \ \ \|f^*_2\|_{L^\infty(\mathbb{R}^+\times \Sigma)}\le Lk_2,$$
 where $\Sigma=[-1,1]$. By using the arguments in the proofs of \cite[Theorem 2.1]{Wjfa16} and \cite[Theorem 2.2]{WZjde18} with some minor modifications, we can get the estimate \eqref{4.1}. This completes the proof.
\end{proof}

\begin{proposition}\lbl{p4.1}\,If $h_\infty-g_\infty<\infty$, then $\dd\lim_{t\rightarrow\infty}g'(t)=\lim_{t\rightarrow\infty}h'(t)=0$.
\end{proposition}

\begin{proof}\, It follows from \eqref{4.3a} that $g'(t)$ and $h'(t)$ are bounded. Let $$\varphi_1(t)=v_x(t,h(t)),\ \ \varphi_2(t)=\int_{g(t)}^{h(t)}\!\!\int_{h(t)}^{\infty}\!J(x-y)v(t,x){\rm d}y{\rm d}x.$$
By \eqref{4.2} we have that, for $t,s>0$,
\[|\varphi_1(t)-\varphi_1(s)|\le C_1|t-s|^{\frac\alpha 2}\]
Using the arguments in the proof of \cite[Proposition 4.1]{DWZ19}, one can show that, there is $C_2>0$ such that
  \[|\varphi_2(t)-\varphi_2(s)|\le C_2|t-s|,\ \ \ \forall\ t,s>0.\]
Note that $h'(t)=-\mu \varphi_1(t)+\rho \varphi_2(t)$. We see that $h'(t)$ is uniformly continuous in $[0,\infty)$. Therefore, $\dd\lim_{t\rightarrow\infty}h'(t)=0$ due to $\dd\lim_{t\rightarrow\infty}h(t)=h_\infty<\infty$. Similarly, $\dd\lim_{t\rightarrow\infty}g'(t)=0$.
\end{proof}

\begin{theorem}\lbl{t4.3}\, If $h_\infty-g_\infty<\infty$, then the solution $(u,v,g,h)$ of \eqref{1.1} satisfies
 $$\lim_{t\to\infty}\|u(t,\cdot)\|_{C([g(t),h(t)])}
 =\lim_{t\to\infty}\|v(t,\cdot)\|_{C([g(t),h(t)])}=0.$$
 \end{theorem}

\begin{proof}\, Noticing $f_2(u,0)=0$, and $f_2(u, v)$ is locally Lipschitz continuous in $u,v\in\mathbb{R}^+$ and $0<u\le k_1$, $0<v\le k_2$, we can write
$f_2(u, v)=\varrho(t,x)v$ with $\varrho\in L^\infty$. Thanks to $h'(t)\ge-\mu v_x(t,h(t))$ and $g'(t)\le-\mu v_x(t,g(t))$. It can be deduced by Lemma \ref{l4.a} that
\bess
\lim_{t\to\infty}\|v(t,\cdot)\|_{C([g(t),h(t)])}=0.
\eess

We next show that
\bes
 \lambda_p(\mathcal{L}_{(g_\infty,h_\infty)}+a)\le0. \label{4.4a}
 \ees
To save spaces, for $\varepsilon>0$ we set
 \[h_\infty^{\pm\varepsilon}=h_\infty\pm\varepsilon, \ \ \ g_\infty^{\pm\varepsilon}=g_\infty\pm\varepsilon.\]
Assume on the contrary that $\lambda_p(\mathcal{L}_{(g_\infty,h_\infty)}+a)>0$. Clearly, there is $\varepsilon_1>0$ such that, for $\varepsilon\in(0,\varepsilon_1)$, $\lambda_p(\mathcal{L}_{(g_\infty^{+\varepsilon},\,h_\infty^{-\varepsilon})}+a-b\varepsilon)>0$. For such $\varepsilon>0$, one can find $T_\varepsilon>0$ such that
 \[h(t)>h_\infty^{-\varepsilon},\ \ \ g(t)<g_\infty^{+\varepsilon},\ \ \|v(t,\cdot)\|_{C([g(t),h(t)])}<\varepsilon,\ \ {\rm for}\ t>T_\varepsilon.\]
Then $u$ satisfies
 \bess
\left\{\begin{aligned}
&u_t\ge d_1\int_{g_\infty^{+\varepsilon}}^{h_\infty^{-\varepsilon}}\!\!J(x-y)u(t,y){\rm d}y-d_1u+u(a-u-b\varepsilon), &&t>T_\varepsilon,~x\in[g_\infty^{+\varepsilon},h_\infty^{-\varepsilon}],\\
&u(T_\varepsilon,x)=u(T_\varepsilon,x), &&x\in[g_\infty^{+\varepsilon},h_\infty^{-\varepsilon}].
\end{aligned}\right.
 \eess
 Consider the problem
\bes
\left\{\begin{aligned}
&w_t=d_1\int_{g_\infty^{+\varepsilon}}^{h_\infty^{-\varepsilon}}\!\!J(x-y)w(t,y){\rm d}y-d_1w+w(a-w-b\varepsilon), &&t>T_\varepsilon,~x\in[g_\infty^{+\varepsilon},h_\infty^{-\varepsilon}],\\
&w(T_\varepsilon,x)=u(T_\varepsilon,x), &&x\in[g_\infty^{+\varepsilon},h_\infty^{-\varepsilon}].
\end{aligned}\right.
 \label{4.4}
 \ees
Since $\lambda_p(\mathcal{L}_{(g_\infty,h_\infty)}+a-b\varepsilon)>0$, it follows from Proposition \ref{p3.4} that the solution $w_\varepsilon(t,x)$ of problem \eqref{4.4} converges to the unique steady state $W_\varepsilon(x)$ of \eqref{4.4} uniformly in $[g_\infty^{+\varepsilon},h_\infty^{-\varepsilon}]$ as $t\rightarrow\infty$. From Lemma \ref{l3.3} and a simple comparison argument, there holds that
\[u(t,x)\ge w_\varepsilon(t,x)\ \ {\rm for}\ \ t>T_\varepsilon\ \ {\rm and}\ \ x\in[g_\infty^{+\varepsilon},h_\infty^{-\varepsilon}].\]
Hence, there is $T_{1\varepsilon}>T_\varepsilon$ such that
 \[u(t,x)\ge \frac12 W_\varepsilon(x)>0\ \ \ {\rm for}\ \ t>T_{1\varepsilon}, \ x\in[g_\infty^{+\varepsilon},h_\infty^{-\varepsilon}].\]

Recall \eqref{1.5}, for $0<\varepsilon<\min\{\varepsilon_1,{\bar\varepsilon}/{2}\}$ and $t>T_{1\varepsilon}$, we obtain
\bess
h'(t)&\ge&\rho\int_{g(t)}^{h(t)}\!\!\int_{h(t)}^\infty\! J(x-y)u(t,x){\rm d}y{\rm d}x\ge\rho\int_{g_\infty^{+\varepsilon}}^{h_\infty^{-\varepsilon}}\!\!\int_{h_\infty}^\infty\! J(x-y)u(t,x){\rm d}y{\rm d}x\\
&\ge&\rho\int_{h_\infty^{-\bar\varepsilon/2}}^{h_\infty^{-\varepsilon}}\int_{h_\infty}^{h_\infty^{+\bar\varepsilon/2}} \delta_0\frac 12W_\varepsilon(x){\rm d}y{\rm d}x>0,
\eess
which implies that $h_\infty=\infty$. We get a contradiction, and so \eqref{4.4a} holds.

Let $\bar u$ be the unique solution of
 \bess\left\{
\begin{aligned}
&\bar u_t=d\int_{g_\infty}^{h_\infty}\!J(x-y)\bar u(t,y)dy-d_1\bar u+\bar u(a-\bar u),&  & t>0,~x\in [g_\infty,h_\infty],\\
&\bar u(0,x)=u_0(x),\ \ |x|\le h_0; \ \ \bar u(0,x)=0, && x\in  [g_\infty,h_\infty]\setminus[-h_0,h_0].
\end{aligned}\right.
 \eess
Using \eqref{4.4a} and Proposition \ref{p3.4} we have $\dd\lim_{t\to\infty}\bar u(t,x)
=0$ uniformly in $[g_\infty,h_\infty]$. As $v\ge 0$, it is clear that $u$ satisfies
  \bess
\left\{\begin{aligned}
&u_t\le d_1\int_{g(t)}^{h(t)}\!J(x-y)u(t,y){\rm d}y-d_1u+u(a-u), &&t>0,~g(t)<x<h(t),\\
 &u(t,g(t))=u(t,h(t))=0,&&t\ge 0,\\
&u(0,x)=u_0(x),\ \ |x|\le h_0.
\end{aligned}\right.
  \eess
Evidently, $\bar u$ satisfies
  \bess
\left\{\begin{aligned}
&\bar u_t\ge d_1\int_{g(t)}^{h(t)}\!J(x-y)\bar u(t,y){\rm d}y-d_1\bar u+\bar u(a-\bar u), &&t>0,~g(t)<x<h(t),\\
 &\bar u(t,g(t))\ge 0,\ \ \bar u(t,h(t))\ge 0,&&t\ge 0,\\
&u(0,x)=u_0(x),\ \ |x|\le h_0.
\end{aligned}\right.
  \eess
Take advantage of Lemma \ref{l2.2} and a comparison argument it can be shown that $u(t,x)\le\bar u(t,x)$ for $t>0$ and $x\in[g(t),h(t)]$. Thus, $\dd\lim_{t\to\infty}\max_{g(t)\le x\le h(t)}u(t,x)=0$. The proof is end.
\end{proof}

\begin{lemma}\label{l4.2}
Let $f_1,f_2$ satisfy \eqref{1a.3}, or \eqref{1a.2} with $1/{c}>a>b$ {\rm(}weak competition{\rm)}. Then $h_\infty-g_\infty=\infty$ if and only if $h_\infty=\infty$ and $g_\infty=-\infty$.
\end{lemma}

The proof of Lemma \ref{l4.2} is similar to those of \cite[Proposition 3.10]{DWZ19}, \cite[Proposition 4.1]{WZjdde17} and \cite[Proposition 3]{WZ-dcdsa18}, and we omit the details.

Here we should mention that if $f_1,f_2$ satisfy \eqref{1a.2} without $1/{c}>a>b$, i.e., the general competition model, we don't know if Lemma \ref{l4.2} is true or not. Even for the local diffusion competition model with double free boundaries
 \bess\left\{\begin{array}{lll}
 u_t-u_{xx}=u(a-u-bv), &t>0,\ \ g(t)<x<h(t),\\[1mm]
  v_t-dv_{xx}=v(1-v-cu),\ \ &t>0, \ \ g(t)<x<h(t),\\[1mm]
 u=v=0,\ \ g'(t)=-\mu_l\big(u_x+\rho_lv_x\big),\ \ &t>0, \ \ x=g(t),\\[1mm]
 u=v=0,\ \ h'(t)=-\mu_r\big(u_x+\rho_rv_x\big),\ \ &t>0,\ \ x=h(t),\\[1mm]
 u(0,x)=u_0(x), \ \ v(0,x)=v_0(x),&x\in [-h_0,h_0],\\[1mm]
 g(0)=-h_0,\ \ h(0)=h_0,
 \end{array}\right.\eess
such a problem is still not clear.

Now we study the long-time behavior of $(u,v)$ when spreading happens.

\begin{theorem}\label{t4.5} Suppose that $h_\infty-g_\infty=\infty$.

{\rm(i)} In the weak competition case we have
\bess
\lim_{t\rightarrow\infty}(u(t,x),v(t,x))=\kk(\df{a-b}{1-bc},\ \df{1-ac}{1-bc}\rr)\ \ {\rm locally\ uniformly\ for}\ x\in\R;
\eess

{\rm(ii)} In the weak predation case we have
\bess
\lim_{t\rightarrow\infty}(u(t,x),v(t,x))=\kk(\df{a-b}{1+bc},\ \df{1+ac}{1+bc}\rr)\ \ {\rm locally\ uniformly\ for}\ x\in\R.
\eess
\end{theorem}
\begin{proof}
It follows from Lemma \ref{l4.2} that $h_\infty=\infty$ and $g_\infty=-\infty$. Similar to the proofs of \cite[Theorem 1.4]{DWZ19} and \cite[Theorem 4.3]{WZjdde17}, by using \cite[Lemma 3.14]{DWZ19},  \cite[Propositions B.1, B.2]{WZjdde17} and iteration arguments, we can get the desired results. To save space, we omit the details.
\end{proof}

\section{The criteria governing spreading and vanishing}
\setcounter{equation}{0} {\setlength\arraycolsep{2pt}

To study the criteria governing spreading and vanishing, we first give two abstract lemmas to affirm that the habitat can be large provided that the moving parameter of free boundary is large enough.

\begin{lemma}{\rm(\cite[Lemma 4.3]{WZ-dcdsa18})}\label{l5.1}\, Let $C$ be a positive constant. For any given positive constants $r_0, H$, and any function $w_0\in W^2_p((-r_0,r_0))$ with $p>1$, $w_0(\pm r_0)=0$ and $w_0>0$ in $(-r_0,r_0)$, there exists $\mu^0>0$ such that when $\mu\geq\mu^0$ and $(w, l, r)$ satisfies
 \bess \left\{\begin{array}{ll}
   w_t-w_{xx}\geq -C w, \ &t>0, \ l(t)<x< r(t),\\[.5mm]
  w=0, \ l'(t)\leq-\mu w_x, \ &t\geq 0, \ x=l(t),\\[.5mm]
 w=0, \ r'(t)\geq-\mu w_x, \ \ &t\geq 0, \ x=r(t),\\[.5mm]
 w(0,x)=w_0(x),\ \ \ &-r_0\leq x\leq r_0,\\[.5mm]
 r(0)=-l(0)=r_0,
  \end{array}\right.
  \eess
we must have $\dd\lim_{t\to\infty}l(t)\le-H$, $\dd\lim_{t\to\infty}r(t)\ge H$.
 \end{lemma}

\begin{lemma}\lbl{l5.2}\,Let the condition {\bf (J)} hold, $d,C>0$ be constants.
For any given constants $H>r_0>0$, and any function $w_0\in C([0,r_0])$
satisfying $w_0(\pm r_0)=0$ and $w_0>0$ in $(-r_0,r_0)$,
there exists $\rho^0>0$, depending on $J(x)$, $d$, $C$, $w_0(x)$ and $r_0$,
such that when $\rho \geq\rho^0$ and $(w, r, l)$ satisfies
 \bess\left\{\begin{aligned}
 &w_t\ge d\dd\int_{l(t)}^{r(t)}\!J(x-y)w(t,y){\rm d}y-d w-Cw, & &t>0,~l(t)<x<r(t),\\
 &w(t,l(t))=w(t,r(t))=0, & &t>0,\\
 &r'(t)\ge\dd\rho \int_{l(t)}^{r(t)}\!\!\int_{r(t)}^{\infty}\!J(x-y)w(t,x){\rm d}y{\rm d}x,\ & &t>0,\\
 &l'(t)\le-\dd\rho \int_{l(t)}^{r(t)}\!\!\int_{-\infty}^{l(t)}\!J(x-y)w(t,x){\rm d}y{\rm d}x, & &t>0,\\
 &w(0,x)=w_0(x),~r(0)=-l(0)=r_0>0, & &|x|\le r_0,
 \end{aligned}\right.
 \eess
we must have $\dd\lim_{t\to\infty}l(t)\le-H$, $\dd\lim_{t\to\infty}r(t)\ge H$.
\end{lemma}

\begin{proof}\,The idea of this proof comes from \cite[Lemma 3.2]{WZjdde14}.
First of all, the comparison principle (\cite[Theorem 3.1 and Lemma 3.3]{CDLL}) gives
 \[r'(t)>0, \ \ l'(t)<0, \ \ w(t,x)>0, \ \ \ \forall \ t>0, \ l(t)<x<r(t).\]
Take a function $\kappa(t)\in C^1([0,1])$ satisfying $\kappa(t)>0$ in $[0,1]$,
$\kappa(0)=r_0$ and $\kappa(1)=H$, and set $\omega(t)=-\kappa(t)$. Consider the following problem
 \bess\left\{\begin{aligned}
&z_t=d\int_{\omega(t)}^{\kappa(t)}\!J(x-y)z(t,y){\rm d}y-dz-Cz, & &0<t<1,~\omega(t)<x<\kappa(t),\\
&z(t,\kappa(t))=z(t,\omega(t))=0,& &0<t<1,\\
&z(0,x)=w_0(x),& &|x|\le r_0.
\end{aligned}\right.
 \eess
In view of Lemma 2.3 of \cite{CDLL}, this problem has a unique solution $z$
which is continuous and positive in $\overline{D}_{1,\omega, \kappa}$. Thus the functions of $t$:
 \begin{align*}
 r(t)=\int_{\omega(t)}^{\kappa(t)}\!\int_{\kappa(t)}^{\infty}\!J(x-y)z(t,x){\rm d}y{\rm d}x,\ \ \
 l(t)=\int_{\omega(t)}^{\kappa(t)}\!\int_{-\infty}^{\omega(t)}\!J(x-y)z(t,x){\rm d}y{\rm d}x\end{align*}
are positive and continuous on $[0, 1]$, and so $r(t), l(t)\ge\sigma>0$
on $[0, 1]$ for some constant $\sigma$. Note that $\omega'(t)$ and $\kappa'(t)$ are bounded
on $[0, 1]$, we can find $\rho^0>0$ such that when $\rho \geq\rho^0$, there hold:
 \begin{align*}
 \kappa'(t)&\;\le\rho r(t)=\rho \int_{\omega(t)}^{\kappa(t)}\!\int_{\kappa(t)}^{\infty}J(x-y)z(t,x){\rm d}y{\rm d}x, \\[1mm]
 \omega'(t)&\;\ge-\rho l(t)=-\rho\int_{\omega(t)}^{\kappa(t)}\!\int_{-\infty}^{\omega(t)}J(x-y)z(t,x){\rm d}y{\rm d}x
 \end{align*}
for all $0\le t\le 1$. Applying the comparison principle (\cite[Theorem 3.1]{CDLL}) we get
 \[l(t)\le \omega(t), \ \ r(t)\ge \kappa(t), \ \ \forall\ 0\le t\le 1,\]
and so $r(1)\ge \kappa(1)=H$ and $l(1)\le \omega(1)=-H$ when $\rho \geq\rho^0$. The desired conclusion is obtained and the proof is complete.
\end{proof}

\begin{theorem}\label{t5.1}\,If $h_\infty-g_\infty<\infty$, then
\bes
h_\infty-g_\infty\le \pi\sqrt{d_2}.\label{5.1}
\ees
\end{theorem}

\begin{proof}
Recall Theorem \ref{t4.3}, $h_\infty-g_\infty<\infty$ implies
\bes
\lim_{t\to\infty}\|u(t,\cdot)\|_{C([g(t),h(t)])}=\lim_{t\to\infty}\|v(t,\cdot)\|_{C([g(t),h(t)])}=0.\label{5.2}
\ees
Assume on the contrary that $h_\infty-g_\infty>\pi\sqrt{d_2}$. Then there exist $0<\varepsilon\ll 1$ and  $\tau\gg1$ such that
 \bess
 &h_\infty-g_\infty-2\varepsilon>\pi\sqrt{d_2(1-c\varepsilon)},&\\[1mm]
 &g(\tau)<g_\infty^{+\varepsilon}, \ \ h(\tau)>h_\infty^{-\varepsilon},&\\[1mm]
 &0\le u(t,x)<\varepsilon,\ \ \forall\ t\ge \tau,\ x\in[g_\infty^{+\varepsilon},h_\infty^{-\varepsilon}].&
 \eess
Then $v$ satisfies
 \bess
\left\{\begin{aligned}
&v_t\ge d_2v_{xx}+v(1-c\varepsilon-v), &&t>\tau,~g_\infty^{+\varepsilon}<x<h_\infty^{-\varepsilon},\\
&v(t,g_\infty^{+\varepsilon})>0, \ \ v(t,h_\infty^{-\varepsilon})>0, &&t\ge\tau,\\
&v(\tau,x)>0, &&g_\infty^{+\varepsilon}<x<h_\infty^{-\varepsilon}.
\end{aligned}\right.
 \eess
Let $w$ be the unique  positive solution of
\bess
\left\{\begin{aligned}
&w_t=d_2w_{xx}+w(1-c\varepsilon-w), &&t>\tau,~g_\infty^{+\varepsilon}<x<h_\infty^{-\varepsilon},\\
&w(t,g_\infty^{+\varepsilon})=w(t,h_\infty^{-\varepsilon})=0, &&t\ge\tau,\\
&w(\tau,x)=v(\tau,x), &&g_\infty^{+\varepsilon}<x<h_\infty^{-\varepsilon}.
\end{aligned}\right.
 \eess
In view of the known parabolic comparison principle, we have
\[w(t,x)\le v(t,x),\ \ \ t\ge\tau,\ g_\infty^{+\varepsilon}\le x\le h_\infty^{-\varepsilon}.\]
Since $h_\infty-g_\infty-2\varepsilon>\pi\sqrt{d_2(1-c\varepsilon)}$, it is well known that $w(t,x)\ri\eta(x)$ as $t\ri\infty$ uniformly in the compact subset of $(g_\infty^{+\varepsilon},h_\infty^{-\varepsilon})$, where $\eta(x)$ is the unique positive solution of
 \bess
\left\{\begin{aligned}
&d_2\eta_{xx}+\eta(1-\eta)=0,\quad x\in(g_\infty^{+\varepsilon},h_\infty^{-\varepsilon}),\\
&\eta(g_\infty^{+\varepsilon})=\eta(h_\infty^{-\varepsilon})=0.
\end{aligned}\right.
 \eess
Hence, $\dd\liminf_{t\ri\infty}v(t,x)\ge\lim_{t\ri\infty}w(t,x)=\eta(x)>0$ for all $x\in (g_\infty^{+\varepsilon},h_\infty^{-\varepsilon})$. This is a contradiction to \eqref{5.2}. Thus, \eqref{5.1} holds.
\end{proof}

From Theorem \ref{t5.1} and $g'(t)<0,h'(t)>0$ for $t>0$, we have

\begin{corollary}\label{c5.2}
If $h_0\ge \frac 12\pi\sqrt{d_2}$, then spreading happens, i.e., $h_\infty-g_\infty=\infty$.
\end{corollary}

 Assume $(f_1,\,f_2)$ satisfies either \eqref{1a.2} or \eqref{1a.3}. If $a\ge d_1$, then $\lambda_p\big(\mathcal{L}_{(0,\ell)}+a\big)>0$
for all $\ell>0$ by Proposition \ref{p3.7}. Thus, the vanishing can not happen
by \eqref{4.4a}, i.e., $h_\infty-g_\infty=\infty$ always holds. Hence, we have

\begin{theorem}\label{t5.3}\,If $a\ge d_1$, then spreading always happens.
\end{theorem}

Now we assume that $a<d_1$. Then, $\lambda_p\big(\mathcal{L}_{(0,\ell)}
+a\big)<0$ if $0<\ell\ll 1$, and $\lambda_p\big(\mathcal{L}_{(0,\ell)}+a\big)>0$
if $\ell\gg 1$ by Proposition \ref{p3.7}. According to the monotonicity of $\lambda_p\big(\mathcal{L}_{(0,\ell)}+a\big)$ with respect to $\ell$, there exists $\ell^*>0$ such that
\[\lambda_p(\mathcal{L}_I+a)=0\ \ {\rm if}\ \ |I|=\ell^*,\ \ \lambda_p(\mathcal{L}_I+a)<0\ \ {\rm if}\ \ |I|<\ell^*,\ \ \lambda_p(\mathcal{L}_I+a)>0\ \ {\rm if}\ \ |I|>\ell^*,\]
where $I$ stands for a finite open interval in $\mathbb{R}$, and $|I|$ denotes its length. Making use of \eqref{4.4a} we see that if $h_\infty-g_\infty<\infty$ then $h_\infty-g_\infty\le\ell^*$. Thus, $h_0\ge \ell^*/2$ implies $h_\infty-g_\infty=\infty$.

\begin{lemma}\label{l5.6}\,Suppose that $a<d_1$.
If $h_0<\frac 12\min\{\pi\sqrt{d_2},\ell^*\}$, then there is $\Lambda_0>0$ such that $h_\infty-g_\infty<\infty$ when $\mu+\rho\le\Lambda_0$.
\end{lemma}

\begin{proof}\,The idea of this proof comes from \cite[Theorem 3.12]{CDLL} and \cite[Lemma 4.4]{WZhang16}. Since $\lambda_p\big(\mathcal{L}_{(-h_0,h_0)}+a\big)<0$, we can choose $h_0<h_1<\ell^*/2$
such that
 \[\lambda:=\lambda_p\big(\mathcal{L}_{(-h_1,h_1)}+a\big)<0.\]

{\it Case 1: The competition model}. That is, $(f_1, f_2)$ satisfies \eqref{1a.2}.
Let $\bar u$ be the unique solution of
 \begin{equation}\label{4.8}
 \left\{\begin{aligned}
&\bar u_t=d_1\int_{-h_1}^{h_1}\!J(x-y)\bar u(t,y){\rm d}y-d_1\bar u+a\bar u,
  & &t>0,~|x|\le h_1,\\
 &\bar u(0,x)=u_0(x), & &  |x|\le h_0, \\
 & \bar u(0,x)=0, & &  |x|>h_0.
 \end{aligned}\right.\end{equation}
And let $\varphi>0$ be the corresponding
normalized eigenfunction of $\lambda$, namely $\|\varphi\|_\infty=1$ and
 \[\left(\mathcal{L}_{(-h_1, h_1)}+a\right)[\,\varphi](x)=\lambda\varphi(x), \ \ \forall\ |x|\le h_1.\]
For $C>0$ and $z(t,x)=C e^{\lambda t/2}\varphi(x)$, it is easy to check that
 \begin{align*}
 d_1\int_{-h_1}^{h_1}\!\!J(x-y)z(t,y){\rm d}y-d_1 z+a z-z_t
 =&\; Ce^{\lambda t/2 }\left(d_1\int_{-h_1}^{h_1}\!\!J(x-y)\varphi(y){\rm d}y-d_1\varphi+a\varphi
 -\frac{\lambda} 2\varphi\right)\\
 =&\;\frac{\lambda} 2C e^{\lambda t/2}\varphi(x)<0, \ \ \forall~t>0, ~|x|\le h_1.
  \end{align*}
Choose $C>0$ large enough such that $C\varphi(x)>u_0(x)$ on $[-h_1, h_1]$.  Then we can apply Lemma \ref{l3.3} to $\bar u-z$ to deduce
 \begin{equation}\label{4.9}
 \bar u(t,x)\le z(t,x)=C e^{\lambda t/2}\varphi(x)\leq C e^{\lambda t/2}, \ \
  \forall~t>0, ~|x|\le h_1.
\end{equation}

Let $0<\delta,\,\sigma<1$ and $K>0$ be constants, which will be determined later. Set
 \bess
 &\dd s(t)=h_0(1+2\delta-\delta {\rm e}^{-\sigma t}), \ \
 \phi(y)=\cos\frac{\pi y}{2},\ \ 0\leq y\leq 1,& \\[1mm]
 &\bar v(t,x)=K{\rm e}^{-\sigma t}\phi\left(x/s(t)\right), \ \ \ t\geq 0,\ \ 0\leq x\leq s(t).&
 \eess
Recall $h_0<\frac 12\pi\sqrt{d_2}$. Similar to the arguments in the proof of \cite[Lemma 3.4]{Wjde15}, we can verify that, for suitably small positive constants $\delta$, $\sigma$, and large positive constant $K$, the pair $(\bar v, s)$ satisfies
 \bes\left\{\begin{array}{ll}\smallskip
 \bar v_t-d_2\bar v_{xx}-\bar v(1-\bar v)\geq 0,\ \ &t>0, \ -s(t)<x\leq s(t),\\[1mm]
 \bar v(0,x)\geq v_0(x),\ \ &-h_0(1+\delta)\leq x\leq h_0(1+\delta).
 \end{array}\right.
 \lbl{5.5}\ees

Set
 \[\bar h(t)=h_0-\mu\int_0^t\bar v_x(\tau,s(\tau)){\rm d}\tau
 +2\rho Ch_1\int_0^t e^{\lambda\tau/2}{\rm d}\tau, \ \ \bar g(t)=-\bar h(t), \ \ \ t\ge0.\]
We claim that if $\mu+\rho\le\Lambda_0$ and $\Lambda_0>0$ is small enough, then $(\bar u, \bar v, \bar g, \bar h)$ is an upper solution of (\ref{1.1}) in there $(f_1, f_2)$ satisfies \eqref{1a.2}. In fact, let
 \[m=\max\left\{\frac {\pi K}{2\sigma h_0(1+\delta)}, \ -\frac{4Ch_1}\lambda\right\}.\]
Noticing
 \[\bar v_x(t,s(t))=-\frac{\pi K}{2s(t)}{\rm e}^{-\sigma t}\]
and $\lambda<0$. It follows that
 \bess
 \bar h(t)&=&h_0+\mu\int_0^t\frac 1{2s(\tau)}\pi K{\rm e}^{-\sigma\tau}{\rm d}\tau
 -\frac 4\lambda\rho Ch_1\left(1-e^{\lambda t/2}\right)\\
 &<&h_0+\mu\frac {\pi K}{2h_0(1+\delta)}\int_0^t{\rm e}^{-\sigma\tau}{\rm d}\tau
 -\frac 4\lambda\rho Ch_1\\
 &<&h_0+\mu\frac {\pi K}{2\sigma h_0(1+\delta)}-\frac 4\lambda\rho Ch_1\leq h_1\eess
provided that
  $$0<\mu+\rho\leq \frac{h_1-h_0}m.$$
Similarly, $\bar g(t)>-h_1 $. In the same way we can show that
 \[\bar h(t)< h_0(1+\delta)\le s(t)\]
provided that
  $$0<\mu+\rho\leq \delta h_0/m.$$
Set
 \[\Lambda_0=\min\left\{\frac{h_1-h_0}m, \ \frac{\delta h_0}m\right\}.\]
Then
 \[\bar h(t)<\min\{h_1,\,s(t)\}, \ \ \ \bar g(t)>\max\{-h_1,\,-s(t)\}\]
provided $\mu+\rho\leq\Lambda_0$. Thus, by \eqref{4.8} and \eqref{5.5} we have
 \[ \bar u_t\ge d\int_{\bar g(t)}^{\bar h(t)}\!J(x-y)\bar u(t,y){\rm d}y
 -d_1\bar u+\bar u(a-\bar u),\ \ t>0,~\bar g(t)<x<\bar h(t)\]
and
  \[\bar v_t- d_2\bar v_{xx}-\bar v(1-\bar v)\geq 0, \ \ t>0, ~\bar g(t)<x<\bar h(t).\]
Due to \eqref{4.9}, it is easy to check that
 \[\int_{\bar g(t)}^{\bar h(t)}\!\!\int_{\bar h(t)}^{\infty}\!J(x-y)\bar u(t,x){\rm d}y{\rm d}x
  \le 2Ch_1e^{\lambda t/2}.\]
On the other hand,
 \bess
 -\bar v_x(t,\bar h(t))=\frac{\pi K}{2s(t)}{\rm e}^{-\sigma t}\sin\frac{\pi \bar h(t)}{2s(t)}\le\frac{\pi K}{2s(t)}{\rm e}^{-\sigma t}.
 \eess
Consequently,
 \bess
 \bar h'(t)&=&-\mu\bar v_x(t,s(t))+2\rho Ch_1 e^{\lambda t/2}=\mu\frac{\pi K}{2s(t)}{\rm e}^{-\sigma t}+2\rho Ch_1 e^{\lambda t/2}\\
 &\ge&-\mu\bar v_x(t,\bar h(t))+\rho\int_{\bar g(t)}^{\bar h(t)}\!\!\int_{\bar h(t)}^{\infty}\!J(x-y)\bar u(t,x){\rm d}y{\rm d}x.
 \eess
Similarly,
  \bess
  \bar g'(t)\le-\mu\bar v_x(t,\bar g(t))-\rho\int_{\bar g(t)}^{\bar h(t)}\!\!\int_{-\infty}^{\bar g(t)}\!J(x-y)\bar u(t,x){\rm d}y{\rm d}x.
  \eess

The above arguments show that $(\bar u, \bar v, \bar g, \bar h)$ is an upper solution of (\ref{1.1}). By Lemma \ref{l3.4}, $g(t)\geq \bar g(t)$, $h(t)\leq\bar h(t)$ for all $t\geq 0$. Therefore
$h_\infty-g_\infty\le 2\dd\lim_{t\to\infty}\bar h(t)\leq 2h_1$.

\sk
{\it Case 2: The prey-predator model}. That is, $(f_1, f_2)$ satisfies \eqref{1a.3}.
Let $h_1$, $\lambda$ and $\varphi$ be as above. Set $\varepsilon=\frac 13\big(\frac{\pi\sqrt{d_2}}{2}-h_0\big)$. It then follows from $h_0<\frac{\pi\sqrt{d_2}}{2}$ that
 \[\frac{d_2\pi^2}{4(h_0+\varepsilon)^2}>1.\]
Take $\sigma$ small such that
 \bes
 c\sigma\le\cos\frac{\pi h_1}{2(h_1+\varepsilon)}.
 \lbl{5.6}\ees
For these fixed $\varepsilon$ and $\sigma$, choosing $k$ large enough such that
\[k\sigma \varphi(x)\ge u_0(x),\ \ k\cos\frac{\pi x}{2(h_0+\varepsilon)}\ge v_0(x)\ \ \ {\rm for}\ x\in[-h_0,h_0].\]
Let
  $$\theta=2\sigma kh_1\rho, \ \ \delta=\frac{k\pi}{2h_0}\mu, \ \ \gamma=\frac 12\min\left\{-\lambda, \ \frac{d_2\pi^2}{4(h_0+\varepsilon)^2}-1\right\}>0.$$
Then, for the fixed $\varepsilon,\sigma,k$ and $\gamma$, there exists $0<\Lambda_0\ll 1$ such that
 \[h_0+\frac{\theta+\delta}\gamma\le h_1,\ \ \frac{d_2\pi^2}{4[h_0+(\theta+\delta)/\gamma+\varepsilon]^2}-\gamma-1>0\]
when $0<\mu+\rho\le\Lambda_0$. We define
 \bess
 &\bar h(t)=h_0+\dd\frac{\theta+\delta}\gamma\big(1-e^{-\gamma t}\big), \ \ \bar g(t)=-\bar h(t),\ \ \ t\ge0,&\\
 &\bar u(t,x)=\dd\sigma k e^{-\gamma t}\varphi(x),\ \ \bar v(t,x)=k e^{-\gamma t}\cos\frac{\pi x}{2(\bar h(t)+\varepsilon)},\ \ t\ge0,\ |x|\le \bar h(t).&
  \eess
Clearly, $\bar h'(t)=(\theta+\delta)e^{-\gamma t}$, $h_0\le \bar h(t)<h_1$, $\bar h(0)=h_0$ and
 \bes
 \frac{d_2\pi^2}{4(\bar h(t)+\varepsilon)^2}-\gamma -1>0.\label{5.4ab}
 \ees
Thanks to \qq{5.6} and $\varphi(x)\le 1$, $\bar h(t)<h_1$, it is not hard to derive
 \bes
c\bar u(t,x)\le\bar v(t,x)\ \ \ {\rm for}\ t>0,\ |x|\le \bar h(t).\label{5.5b}
 \ees
The choices of $\varepsilon,\sigma$ and $k$ guarantee that
\bess
\bar u(0,x)\ge u_0(x),\ \ \ \bar v(0,x)\ge v_0(x)\ \ \ {\rm for}\ \ |x|\le h_0.
\eess
Moreover, it is easy to see that
 \[\bar u(t,\pm \bar h(t)), \ \bar v(t,\pm \bar h(t))\ge 0.\]

It is easy to deduce that, for $t>0$ and $|x|\le \bar h(t)$,
  \bess
\bar u_t-d_1\int_{-h_1}^{h_1}\!J(x-y)\bar u(t,y){\rm d}y+d_1\bar u-a\bar u
&=&\sigma k e^{-\gamma t}\big(-\gamma \varphi-(\mathcal{L}_{(-h_1,h_1)}+a)[\varphi]\big)\\
&=&\sigma k e^{-\gamma t}\big(-\gamma \varphi-\lambda\varphi\big)\\
&\ge&\sigma k e^{-\gamma t}\gamma \varphi>0.
 \eess
Consequently,
 \bess
 \bar u_t&\ge& d_1\int_{-h_1}^{h_1}\!J(x-y)\bar u(t,y){\rm d}y-d_1\bar u+\bar u(a-\bar u)\\
 &\ge&d_1\int_{\bar g(t)}^{\bar h(t)}\!J(x-y)\bar u(t,y){\rm d}y-d_1\bar u+\bar u(a-\bar u), \ \ t>0, \ |x|\le \bar h(t).
 \eess

Writing $y=\frac{\pi x}{2(\bar h(t)+\varepsilon)}$. Then $\frac{\sin y}{\cos y}x\ge 0$
for $|x|\le \bar h(t)$. By direct calculations, we have, for $t>0$ and $x\in[\bar g(t),\bar h(t)]$,
 \bess
\bar v_t(t,x)&=&-\gamma \bar v+k e^{-\gamma t}\frac{\pi xr'(t)}
 {2(\bar h(t)+\varepsilon)^2}\sin y\\
 &=&-\gamma\bar v+\bar v \frac{(\theta+\delta)\pi e^{-\gamma t}}{2(\bar h(t)+\varepsilon)^2}\frac{\sin y}{\cos y}x\\
&\ge&-\gamma \bar v,\\
 \bar v_{xx}(t,x)&=&-\frac{\pi^2}{4(\bar h(t)+\varepsilon)^2}\bar v.
 \eess
Recall \eqref{5.4ab} and \eqref{5.5b}. It follows that, for $t>0$ and $|x|\le \bar h(t)$,
 \bess
 \bar v_t-d_2\bar v_{xx}-\bar v(1-\bar v+c\bar u)
&\ge&\bar v\left(-\gamma +\frac{\pi^2}{4(\bar h(t)+\varepsilon)^2}-1+\bar v-c\bar u\right)\\
&\ge&\bar v\left(-\gamma +\frac{\pi^2}{4(\bar h(t)+\varepsilon)^2}-1\right)\\
&\ge&0.
 \eess
It is easy to verify that, for $t\ge0$,
 \bess
 &\dd\rho\int_{\bar g(t)}^{\bar h(t)}\!\!\int_{\bar h(t)}^\infty\!J(x-y)\bar u(t,y){\rm d}y{\rm d}x\le 2\rho\sigma kh_1e^{-\gamma t}=\theta e^{-\gamma t},&\\
 &\dd-\mu\bar v_x(t,\bar h(t))=\frac{\mu k\pi}{2(\bar h(t)+\varepsilon)}\sin\frac{\pi \bar h(t)}{2(\bar h(t)+\varepsilon)}e^{-\gamma t}\le\frac{\mu k\pi}{2h_0}e^{-\gamma t}=\delta e^{-\gamma t}.&
 \eess
It follows that
\bess
\bar h'(t)=(\theta+\delta)e^{-\gamma t}\ge-\mu\bar v_x(t,\bar h(t))+\rho\int_{\bar g(t)}^{\bar h(t)}\!\!\int_{\bar h(t)}^\infty \!J(x-y)\bar u(t,y){\rm d}y{\rm d}x.
\eess
Similarly, one has
 \bess
\bar g'(t)\le-\mu\bar v_x(t,\bar g(t))-\rho\int_{\bar g(t)}^{\bar h(t)}\!\!
\int_{-\infty}^{\bar g(t)}\!J(x-y)\bar u(t,y){\rm d}y{\rm d}x.
 \eess

Above all, we conclude that $(\bar u,\bar v,\bar g, \bar h)$ is an upper solution of \eqref{1.1}.
By Lemma \ref{l3.5}, $h(t)\le \bar h(t)$, $g(t)\ge \bar g(t)$.
Therefore, $h_\infty-g_\infty\le 2\dd\lim_{t\ri\infty}\bar h(t)\le 2h_1<\infty$.
This completes the proof.
\end{proof}

\begin{theorem}\label{t5.4}\,Suppose that $a<d_1$.

{\rm(i)}\, If $h_0\ge\frac 12\min\{\pi\sqrt{d_2},\ell^*\}$ then spreading always happens;

{\rm(ii)}\, If $h_0<\frac 12\min\{\pi\sqrt{d_2},\ell^*\}$, then there is $\Lambda^*\ge\Lambda_*>0$ such that $h_\infty-g_\infty<\infty$ when $\mu+\rho\le\Lambda_*$, and $h_\infty-g_\infty=\infty$ when $\mu+\rho\ge\Lambda^*$.
\end{theorem}

\begin{proof}\,(i)\,If $h_0\ge\frac 12\pi\sqrt{d_2}$, then spreading happens by Corollary \ref{c5.2}.
If $h_0\ge{\ell^*}/2$ and vanishing happens, then $[g_\infty,h_\infty]$ is a finite interval with length strictly bigger than $2h_0\ge\ell^*$. Hence $\lambda_p(\mathcal{L}_{(g_\infty,h_\infty)}+a)>0$. This contradicts \eqref{4.4a}.

(ii)\,As $u,v$ are bounded, there exists constant $C>0$ such that $f_1(u,v)\ge -Cu$, $ f_2(u,v)\ge -Cv$. Clearly,
 \bess
 \begin{aligned}
 &h'(t)>-\mu v_x(t,h(t)), \ \ h'(t)>\rho\int_{g(t)}^{h(t)}\!\!\int_{h(t)}^\infty\!J(x-y)u(t,x){\rm d}y{\rm d}x, &&t\ge 0,\\[1mm]
&g'(t)<-\mu v_x(t,g(t)), \ \ g'(t)<-\rho\int_{g(t)}^{h(t)}\!\!\int_{-\infty}^{g(t)}\!J(x-y)u(t,x){\rm d}y{\rm d}x, &&t\ge 0,\\
 \end{aligned}\eess
Fixed a constant $H>\min\{\pi\sqrt{d_2},\ell^*\}$ and let $\mu^0$ and $\rho^0$ be obtained by Lemma \ref{l5.1} and Lemma \ref{l5.2}, respectively, and set $\Lambda^0=\mu^0+\rho^0$. Then $h_\infty-g_\infty=\infty$ when $\mu+\rho>\Lambda^0$ by
Lemma \ref{l5.1}, Lemma \ref{l5.2} and the conclusion (i). Let $\Lambda_0$ be given by Lemma \ref{l5.6}. Then $h_\infty-g_\infty<\infty$ when $\mu+\rho\le\Lambda_0$.

By use of the continuity method: increasing $\Lambda_0$ and decreasing $\Lambda^0$ continuously, similar to the arguments of \cite[Theorem 5.2]{WZjdde17}, we can show the desired conclusions and the details are omitted here. This completes the proof.
\end{proof}

\begin{proof}[Proof of Theorem \ref{th1.3}]
The statements in Theorem \ref{th1.3} are contained in Corollary \ref{c5.2}, Theorem \ref{t5.3} and Theorem \ref{t5.4}.
\end{proof}


\begin{thebibliography}{99}
\bibliographystyle{siam}
\setlength{\baselineskip}{15pt}

\vspace{-1mm}\bibitem{WW-I}J. P. Wang and M. X. Wang, {\it Free boundary problems with nonlocal and local diffusions I: global solution}. arXiv:1812.11643v2.

\vspace{-1mm}\bibitem{GW12}J. Guo and C. H. Wu, {\it On a free boundary problem for a two-species weak competition system}, J. Dyn. Diff. Equat., \textbf{24}
(2012) 873-895.

\vspace{-1mm}\bibitem{WZjdde14} 
M. X. Wang and J. F.  Zhao, {\it Free boundary problems for a Lotka-Volterra competition system}, {J. Dyn. Diff. Equat.}, \textbf{26}(3)(2014), 655-672.

\vspace{-1mm}\bibitem{Wcnsns15} M.X. Wang, {\it Spreading and vanishing in the diffusive prey-predator model with a free boundary},  Commun. Nonlinear Sci. Numer. Simulat., {\bf 23}(2015), 311-327.

\vspace{-1mm}\bibitem{DWZ19}Y. H. Du, M. X. Wang and M. Zhao, {\it Two species nonlocal diffusion systems with free boundaries}. arXiv:1907.04542v1.

\vspace{-1mm}\bibitem{CDLL}J. F. Cao, Y. H. Du, F. Li and W. T. Li, {\it The dynamics of a
Fisher-KPP nonlocal diffusion model with free boundaries}, J. Funct. Anal., {\bf 277}(2019), 2772-2814.

\vspace{-1mm}\bibitem{WZhang16} M. X. Wang and Y. Zhang,
{\it The time-periodic diffusive competition models with a free boundary and sign-changing growth rates}, Z. Angew. Math. Phys. {\bf 67}(5)(2016): 132. DOI: 10.1007/s00033-016-0729-9.

\vspace{-1mm}\bibitem{WZjdde17} M. X. Wang and J. F. Zhao, {\it A free boundary problem for the  predator-prey model with double free boundaries}, J. Dyn. Diff. Equat., {\bf 29}(3)(2017), 957-979.

\vspace{-1mm}\bibitem{7-BJFA16}H. Berestycki, J. Coville and H. Vo, {\it On the definition and the properties of the principal eigenvalue of some nonlocal operators}, J. Funct. Anal., \textbf{271}
(2016), 2701-2751.

\vspace{-1mm}\bibitem{10-Bjmaa07}P. Bates and G. Y. Zhao, {\it Existence, uniqueness and stability of the stationary solution to a nonlocal evolution equation arising in population dispersal}, J. Math. Anal. Appl., \textbf{332} (2007), 428-440.

\vspace{-1mm}\bibitem{19-Cjde10}J. Coville, {\it On a simple criterion for the existence of a principal eigenfunction of some nonlocal operators}, J. Differential Equations, \textbf{249} (2010), 2921-2953.

\vspace{-1mm}\bibitem{WZ-dcdsa18} M. X. Wang and Q. Y. Zhang, \emph{Dynamics for the diffusive leslie-gower model with double free boundaries}, Discrete Cont. Dyn. Syst., {\bf 38}(5)(2018),  2591-2607.

\vspace{-1mm}\bibitem{Wjfa16} M.X. Wang, {\it A diffusive logistic equation with a free boundary
and sign-changing coefficient in time-periodic environment}, J. Funct. Anal., \textbf{270} (2016),  483-508.

\vspace{-1mm}\bibitem{WZjde18} M. X. Wang and Y. Zhang,
\newblock {\it Dynamics for a diffusive prey-predator model with different free boundaries}, J. Differental Equatons, \textbf{264} (2018), 3527-3558.

\vspace{-1mm}\bibitem{Wjde15} 
M. X. Wang, {\it The diffusive logistic equation with a free boundary and sign-changing coefficient}, {J. Differential Equations}, \textbf{258}(4)(2015), 1252-1266.

\end{thebibliography}
\end{document}